\documentclass[letterpaper, 10 pt, conference]{ieeeconf}
\usepackage{graphicx}
\usepackage{subfigure}
\usepackage[usenames,dvipsnames,svgnames,table]{xcolor}
\definecolor{steelblue}{RGB}{70,130,180}
\definecolor{warmred}{RGB}{244, 104, 65}
\usepackage{amsmath}
\usepackage{amssymb,amsfonts,mathrsfs,mathtools}
\usepackage{epstopdf}
\usepackage{lscape}
\usepackage{cite}
\usepackage{doi}
\usepackage{hyperref}
\IEEEoverridecommandlockouts

\newtheorem{remark}{Remark}

\newtheorem{theorem}{Theorem}
\newtheorem{corollary}{Corollary}

\allowdisplaybreaks[4]
\begin{document}
\title{Stochastic Dynamic Programming for Wind Farm Power Maximization}
\author{Yi Guo, \quad
        Mario Rotea, \quad
          Tyler Summers% <-this % stops a space
    \thanks{This material is based on work supported by the National Science Foundation under grant CMMI-1728605. {This project was partially funded by The University of Texas at Dallas Office of Research through the SCI program}.}
    \thanks{Y. Guo, M. Rotea and T. Summers are with the Department of Mechanical Engineering, The University of Texas at Dallas, Richardson, TX, 75080, USA, email: \{yi.guo2, rotea, tyler.summers\}@utdallas.edu.}
}
\maketitle

\begin{abstract}
Wind plants can increase annual energy production with advanced control algorithms by coordinating the operating points of individual turbine controllers across the farm. It remains a challenge to achieve performance improvements in practice because of the difficulty of utilizing models that capture pertinent complex aerodynamic phenomena while remaining amenable to control design. We formulate a multi-stage stochastic optimal control problem for wind farm power maximization and show that it can be solved analytically via dynamic programming. In particular, our model incorporates state- and input-dependent multiplicative noise whose distributions capture stochastic wind fluctuations. The optimal control policies and value functions explicitly incorporate the moments of these distributions, establishing a connection between wind flow data and optimal feedback control. We illustrate the results with numerical experiments.% that demonstrate the advantages of our approach over existing methods based on deterministic models. 
\end{abstract}

\section{Introduction}
Wind energy is an important component of future energy systems to meet growing energy demands. As wind power continues to account for a larger portion of the world-wide energy portfolio, the optimal operation of wind farms offers both challenges and opportunities to further improve performance at the levels of single turbines, wind farms, and power grids. Due to nonlinear aerodynamic interaction through wakes and unpredictable wind variations, future optimal control strategies for wind farms will require sophisticated models to capture and manage \emph{stochastic} wind fluctuations.

Maximizing the wind power capture has been discussed in the scope of wind turbines \cite{pao2009tutorial,xiao2018cart3,pao2011control,munteanu2008optimal} and wind farms \cite{schepers2007improved,johnson2009wind,knudsen2009distributed,spudic2010hierarchical,madjidian2011distributed,soleimanzadeh2011controller,kristalny2011decentralized,horvat2012quasi,biegel2013distributed,bitar2013coordinated,buccafusca2018maximizing,buccafusca2017modeling,ciri2017model,santhanagopalan2018performance,gebraad2016wind,goit2015optimal,gebraad2015maximum,rotea2014dynamic,johnson2012assessment,yang2013maximizing,marden2013model,park2013wind}. In Region 2 operation (below-rated wind speed), the wind plant is operated to maximize the power output. In this regime, there are inherent tradeoffs between the wake of upstream turbines and the power extracted from downstream turbines. Due to this aerodynamic coupling, maximizing total power of wind farms cannot be achieved by myopically maximizing the power output for each individual wind turbine in the array \cite{steinbuch1988optimal}. Therefore, depending on layout and wind conditions, it may be essential to have a coordinated control framework for wind farms to determine the optimal control strategy for each wind turbine to improve annual energy production.

Many challenges and related solutions for wind farm power maximization have been highlighted and discussed in \cite{pao2009tutorial}. Recent control strategies for optimal operation have been proposed using  both model-based \cite{bitar2013coordinated,gebraad2016wind,goit2015optimal,rotea2014dynamic,johnson2009wind,knudsen2009distributed,spudic2010hierarchical,madjidian2011distributed,soleimanzadeh2011controller,kristalny2011decentralized,horvat2012quasi,biegel2013distributed}, and model-free strategies \cite{ciri2017model,marden2013model,park2013wind,gebraad2015maximum,yang2013maximizing}. 
%Some of these works \cite{biegel2013distributed,spudic2010hierarchical} focus on minimizing the cumulative loads for all turbines to achieve the larger wind power extraction, which assumes that the wind power is always greater than the total power demands. 
%The models in \cite{johnson2009wind,bitar2013coordinated} provide rigorous frameworks to maximize the wind power extraction in wind farms directly. 
Model-based strategies provide solutions that typically have faster response times than model-free approaches. However, the models used for control design can deviate from actual wind field and turbine characteristics in practice, which can limit the effectiveness of model-based control strategies. The reader is referred to the introduction in  \cite{ciri2017model}, and the references therein, for further discussion on model-based and model-free strategies for wind plant power maximization.

% Literature review
In this paper, we focus on wind power maximization in Region 2. The work presented here generalizes the simple actuator disk model (ADM) utilized in \cite{rotea2014dynamic} to a stochastic version and pose a multi-stage stochastic optimal control problem for wind farm power maximization. The stochastic actuator disk model balances complexity and tractability by incorporating unsteady aerodynamic phenomena into the distributions of random variables in the model. Estimates of the statistics of these distributions can then be exploited in the control algorithm to improve overall efficiency of the farm in the presence of stochastic wind flow.

Our main contributions are as follows:
        \begin{itemize}
        \item[-] We formulate a multi-stage stochastic optimal control problem for wind farm power maximization and show that it can be analytically solved via dynamic programming. In particular, our model generalizes that of \cite{rotea2014dynamic} by incorporating state- and input-dependent multiplicative noises to capture the uncertain wake effects of wind turbines. The stochastic version of the ADM relaxes a strong assumption of a deterministic ADM, such as steady wind over the rotor disk. In contrast to existing work, the proposed stochastic multi-stage formulation allows us to maximize the wind farm power by explicitly incorporating information about the probability distributions of wind fluctuations into control decisions.
        
        \item[-] By solving the proposed multi-stage stochastic optimization, we show that the optimal feedback control policies for the turbines are linear with respect to upstream wind velocity, but in contrast to \cite{rotea2014dynamic}, the optimal gain coefficients depend explicitly on the statistics of the multiplicative noises, which can be estimated from high-fidelity wind flow simulations or experimental data. This provides a direct connection between statistical properties of the unsteady wind flow physics and the optimal feedback control of wind farms. We also show that for the stochastic ADM with both multiplicative and additive noise, the optimal policies are nonlinear. 
        
\end{itemize}

The framework, while elementary for real-world applications, illustrates a rigorous process for incorporating flow statistics into the wind farm power maximization problem. The dependence of control solutions on the statistics of the wind fluctuations makes intuitive sense, as one cannot expect a single control algorithm to be optimal under a range of unsteady wind conditions. In future work, we will extend the stochastic approach presented in this paper to more representative, yet tractable, models of the flow physics and loads as done in~\cite{santhanagopalan2018performance}.

\section{Problem Formulation}
Our model is a generalization of the one in \cite{rotea2014dynamic}, which utilizes the actuator disk model (ADM) \cite{burton2001handbook,manwell2010wind}. Let $P$ denote the power extracted by an ideal turbine rotor, let $F$ denote the force done by the wind on the rotor, let $V_0$ denote the free stream upwind velocity, let $V$ denote the wind velocity at the disk, and let $V_1$ denote the far wake velocity. The ADM model is then
\begin{subequations}\label{physicalModels}
\begin{equation}
P=FV,\label{Power_def}
\end{equation}
\begin{equation}
F = \rho A(V_0-V_1)V,\label{forces_def}
\end{equation}
\begin{equation}
V = V_0 - u, \label{y_def}
\end{equation}
\begin{equation}
V_1 = V_0 - 2u, \label{xk1_def}
\end{equation}
\end{subequations}
where $\rho$ is the air density, $A$ is the rotor swept area, and $u\geq 0$ is the reduction in air velocity between the free stream and the rotor plane, which can be interpreted as a control input. In practice, $u$ can be controlled by adjusting the angular rotor speed or the collective blade pitch angle.

\begin{figure}[!htbp]
\centering
\includegraphics[scale=1.3]{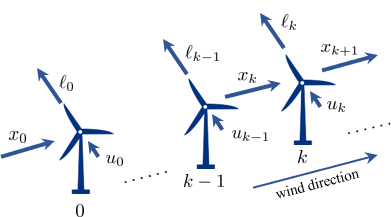}
\caption{A cascade of $N$ wind turbines; $k=0$ indicates the most upstream location.}
\label{fig:wind_farm_array}
\end{figure}

\textbf{Deterministic Model:} We consider a one-dimensional cascade of wind turbines, illustrated in Fig.~1. We assume that the wind direction is along the row of turbines and is not varying. The ADM model given in \eqref{physicalModels} can be written in state-space form by letting $x_k$ and $x_{k+1}$ denote the wind velocity upstream and downstream of the $k$-th turbine (i.e., $x_k = V_k$ in \eqref{xk1_def}, for $k=0,1$). The scalar control input for the $k$-th turbine is denoted by $u_k$, which is the controllable wind velocity deficit at the rotor disk, and $y_k$ is an output to estimate the power extraction of turbine $k$ (i.e., $y_k = V_k - u_k$ in \eqref{y_def}). Then the velocity $V_{k+1}$ in the far wake of the rotor \eqref{xk1_def} and the rotor effect at the disk in velocity \eqref{y_def} can be written as below in \eqref{SS_Velocity} and \eqref{SS_output}. The power extraction of the $k$-th wind turbine using ADM model \eqref{physicalModels} in state-space expression is given in \eqref{SS_power}
        \begin{subequations}
        \begin{align}
            x_{k+1} & = x_k - 2u_k, \label{SS_Velocity}\\
            y_k & = x_k - u_k, \label{SS_output}\\
            p_k(y_k, u_k) & = 2\rho A y_k^2 u_k, \label{SS_power}
        \end{align}
        \end{subequations}
        where the control  input  is  constrained  by $u_k\in[0,\frac{1}{2}x_k]$ so that the wind velocity in the far wake remains positive.  To simplify the notation, we eliminate the constant in \eqref{SS_power} and come to the constant-free turbine power function $\ell(x_k, u_k)$, which will serve as a stage cost in our subsequent multi-stage optimal control problem
        \begin{equation} \label{SS_togofunctions}
        \ell(x_k, u_k) = (x_k - u_k)^2 u_k.
        \end{equation}
        Note that this function is jointly \emph{cubic} in the state and control input. Further details of this model may be found in \cite{rotea2014dynamic}.

\textbf{Stochastic Model:}
The simple model described above captures basic wind farm turbine interactions. But it fails to capture stochastic wind fluctuations that are also relevant to optimizing the total power output. High fidelity computational fluid dynamic models offer extreme detail of flows but are cumbersome to incorporate into high-level operational decision making. Therefore, we consider here a stochastic extension of the deterministic actuator disk model above that can capture more complex phenomena, such as stochastic wind fluctuations, while remaining computationally tractable.

\begin{figure}[!htbp]
\centering
\includegraphics[scale=0.55]{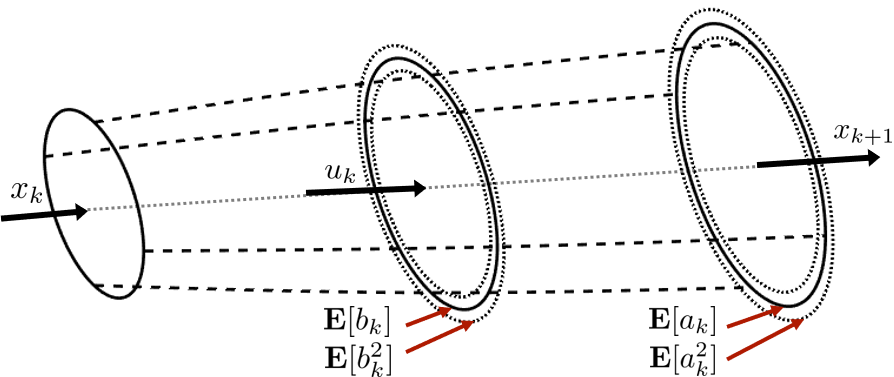}
\caption{Stochastic actuator disk model and stream-tube diagram for wind power extraction. The solid and dashed lines indicate the wind field mean and associated stochastic variations, respectively, which relate to the moments of the multiplicative variations parameters $a_k$ and $b_k$.}
\label{fig:stochasticADM}
\end{figure}

\begin{figure*}[!htbp]
\hrulefill
\begin{subequations}
\begin{align}
    \psi_k & = -\frac{3Q_{k+1}\Sigma_{b,k}\mu_{a,k} -2  + \sqrt{(3Q_{k+1}\Sigma_{b,k}\mu_{a,k} -2 )^2 - 3(Q_{k+1}\Gamma_{b,k}+1)(3Q_{k+1}\Sigma_{a,k}\mu_{b,k}+1)}} {3(Q_{k+1}\Gamma_{b,k}+1)},\label{optimal_feedback_policies_raw}\\
        Q_k & = (1-\psi_k)^2\psi_k + Q_{k+1}\left(\Gamma_{a,k} + \Gamma_{b,k}\psi_k^3 + 3\Sigma_{b,k}\mu_{a,k}\psi_k^2 + 3\Sigma_{a,k}\mu_{b,k}\psi_k \right).\label{back_recursion_raw}
\end{align}
\end{subequations}
\hrulefill
\end{figure*}

The stochastic actuator disk model is given by
\begin{equation} \label{stochmodel}
x_{k+1} = a_k x_k + b_k u_k,% + c_k,
\end{equation}
where $a_k \sim \mathcal{P}_{a,k}$ is a state multiplicative random variable and $b_k \sim \mathcal{P}_{b,k}$ is an input multiplicative random variable. %and $c_k \sim \mathcal{P}_c$ is an additive random variable. 
The model is illustrated in Fig.~\ref{fig:stochasticADM}. We assume that the random variables $a_k$ and $b_k$ are independent for all $k$ and independent of each other.
This model captures stochastic wind fluctuations. In particular, the multiplicative noises $a_k$ and $b_k$ provide a simple model for the inherent stochasticity of far wake recovery. We assume that moments up to order three of each of the distributions $\mathcal{P}_{a,k}, \mathcal{P}_{b,k}$ are known (or can be estimated from high-fidelity simulation or experimental data). For the state mean dynamics to match the deterministic model \eqref{SS_Velocity}, we can set $\mathbf{E} [a_k] = 1$, $\mathbf{E} [b_k] = -2$. 

\section{Stochastic Optimal Control for Wind Power Maximization}
The objective of the operator is to select control inputs $u_0,...,u_{N-1}$ to maximize the aggregate power of the wind turbine cascade given by the sum of \eqref{SS_togofunctions} over all turbines. However, since in the stochastic model the states (and therefore the power outputs) are random variables, we maximize the \emph{expected} aggregate power and search for closed-loop feedback control policies that specify control inputs as a function of the state $x_k$. In particular, we seek to solve the multi-stage stochastic optimal control problem
\begin{equation} \label{stochcost}
\max_{\pi_0, ..., \pi_{N-1}} \quad \mathbf{E} \sum_{k=0}^{N-1} (x_k - u_k)^2 u_k,
\end{equation}
where the decision variables $\pi_k(\cdot)$ are the control policies (i.e., $u_k = \pi_k(x_k)$), and the expectation is taken with respect to the random variable sequences $a_k$, $b_k$. As in \cite{rotea2014dynamic}, we will show that the optimal policies are linear and the optimal value functions are \emph{cubic}. In contrast to \cite{rotea2014dynamic}, the parameters of both the optimal policies and value functions depend on the moments of the distribution of the random variables in the model. We have the following main result.

\begin{theorem}
Consider the wind farm power maximization problem for a cascade of $N$ identical turbines modeled with the stochastic actuator disk model \eqref{stochmodel}, \eqref{stochcost}. Let $x_0$ denote the free stream velocity entering the cascade. The distributions of $a_k$ and $b_k$ are described by their raw moments up to third order, namely their
means $\mu_{a,k}$, $\mu_{b,k}$, second (raw) moments $\Sigma_{a,k}$, $\Sigma_{b,k}$ and third (raw) moments $\Gamma_{a,k}$, $\Gamma_{b,k}$. Under these assumptions, the optimal feedback control policies are linear in the state and given by
\begin{equation}\label{optimal_linear_policies_theorem}
u_k^* = \pi^*(x_k) =  \psi_k x_k, \quad k = 1,...,N-1,
\end{equation}
where the gain coefficients $\psi_k$ are given in \eqref{optimal_feedback_policies_raw}
and the backwards recursion \eqref{back_recursion_raw} for $k = N-1,...,0$ with initialization $Q_N = 0$. The maximum power produced by the wind farm as a function of initial upstream wind velocity is given by
\begin{equation}\label{maximum_wind_farm_power}
P_0^*(x_0) = 2\rho A Q_0 x_0^3,
\end{equation}
where $\rho$ is the air density, $A$ is the rotor swept area, and $Q_0$ is the initial value of the backwards recursion \eqref{back_recursion_raw} with $Q_N = 0$.
\end{theorem}

\begin{proof}
The dynamic programming algorithm \cite{bellman1954theory,bertsekas2005dynamic} for solving stochastic optimal control problems is given by the recursion
\begin{equation} \label{DP_algorithm}
\begin{aligned}
    G_k^*(x_k) &= \max_{u_k \in [0, \frac{1}{2}x_k]} \mathbf{E}\left\{ \ell(x_k, u_k) + G_{k+1}^*\left(x_{k+1})\right)  \right\}, \\
    \pi^*(x_k) &= \arg \max_{u_k \in [0, \frac{1}{2}x_k]} \mathbf{E}\left\{ \ell(x_k, u_k) + G_{k+1}^*\left(x_{k+1})\right)  \right\},
\end{aligned}
\end{equation}
where $G_k^*(x_k)$ represents the optimal (normalized) wind farm power from turbine $k$ as a function of the state $x_k$, with initialization $G_N^*(x_k) = 0$. We first solve the last tail sub-problem at $k=N-1$ with $G_N^*(x) = 0$. We have
\begin{equation}\nonumber
\frac{\partial \ell(x_{N-1}, u_{N-1})}{\partial u_{N-1}} = (x_{N-1} - u_{N_1})(x_{N-1} - 3u_{N-1}) = 0,
\end{equation}
for which the policy $u_{N-1}^* = \frac{1}{3}x_{N-1}$ is the unique maximizer and satisfies the constraint $u_{N-1} \in [0,\frac{1}{2}x_{N-1}]$. Substituting this optimal policy back into the value expression yields the optimal power function
\begin{equation*}
    G_{N-1}^*(x_{N-1}) = \frac{4}{27}x_{N-1}^3.
\end{equation*}
Note that this function is a cubic in the state. Accordingly, we parameterize the optimal power functions as $G_{k}^*(x_{k}) = Q_{k}x_{k}^3$ and consider a general step in the backward recursion. To obtain the optimal policy, we define the function inside the maximization operation
\begin{equation}\label{cost-to-go function}
    G_k(x_k,u_k):= (x_k - u_k)^2u_k + Q_{k+1}\mathbf{E}\left[(a_kx_k + b_ku_k)^3\right].
\end{equation}
Expanding the second term and taking the expectation by utilizing the (raw) moment information from the distributions of $a_k$ and $b_k$, and then taking the partial derivative of $G_k(x_k,u_k)$ with respect to $u_k$ gives a quadratic polynomial in $u_k$. As above, one of the roots of this polynomial corresponds to the unique maximizing input, which is a \emph{linear} function of the state. Carrying out the algebra yields
\begin{equation}\label{state_feedback_control_policies}
    \begin{aligned}
        u_k^* = \pi^*(x_k) = \psi_k x_k,
    \end{aligned}
\end{equation}
where the gain parameters $\psi_k$ are given in \eqref{optimal_feedback_policies_raw}. Note that the optimal policies all satisfy the constraints on $u_k$. To obtain a backwards recursion for the value function coefficients $Q_{k}$, we substitute $u_k^* = \psi_kx_k$ back into \eqref{cost-to-go function}
\begin{equation}
\begin{aligned}\label{optimal_recursion}
    G_k^*(x_k,u_k^*) = & ~ Q_{k}x_k^3\\
    = & (x_k - u_k^*)^2u_k^* + Q_{k+1}\mathbf{E}\left[(a_kx_k + b_ku_k^*)^3\right].
    \end{aligned}
\end{equation}
Since $u_k^*$ is linear in $x_k$, the optimal value functions are cubic in the state.  Matching the coefficients on both sides of \eqref{optimal_recursion}, we come to \eqref{back_recursion_raw}. \color{black}Eq. \eqref{maximum_wind_farm_power} follows from \eqref{optimal_recursion} for $k=0$, \eqref{stochcost} and \eqref{SS_power}, which concludes the proof.\color{black}
\end{proof}

\begin{figure*}[!htbp]
\hrulefill
\begin{equation}\label{nonlinear_control_policies}
\begin{aligned}
    &\pi_{N-2}^*(x) = -\frac{\Delta_k  + \sqrt{\Delta_{N-2}^2 - 3(Q_{N-1}\Gamma_{b,N-2}+1)\big[\left(3Q_{N-1}\Sigma_{a,N-2}\mu_{b,N-2}+1\right)x^2 + 3Q_{N-1}\Sigma_{c,N-2}\mu_{b,N-2}\big]}} {3(Q_{N-1}\Gamma_{b,N-2}+1)},\\
    & \textrm{where,}\quad\Delta_{N-2} = (3Q_{N-1}\Sigma_{b,N-2}\mu_{a,N-2} - 2) x.
\end{aligned}
\end{equation}
\hrulefill
\end{figure*}
\begin{remark} (Optimal policies and value functions with central moments.) The random variables $a_k$ and $b_k$ can also be described by their higher-order central moments, namely their variances $\sigma_{a,k}^2$, $\sigma_{b,k}^2$ and skewnesses $\gamma_{a,k}$, $\gamma_{b,k}$. The optimal linear state feedback control policies can also be written in terms of central moments instead of raw moments by using
\begin{equation} \label{raw_central_moments}
\Sigma = \sigma^2 + \mu^2, \quad \Gamma = \sigma^3\gamma + 3\sigma^2\mu + \mu^3.
\end{equation}
%We do not demonstrate the optimal control policies $\psi_k$ and backward recursion constant $Q_k$ in central moments for space saving purpose.
\end{remark} 
\begin{corollary}
Under the assumptions of Theorem 1, we define the efficiency $\eta_\ell$ of the $\ell$-th sub-array\footnote{The efficiency $\eta_\ell$ defined here quantifies the energy extraction of sub-array $\ell$ compared to energy in the wind entering the sub-array. Note that due to aerodynamic wake coupling, it is possible for the optimal efficiency of the sub-array to exceed the efficiency obtained by independently setting individual turbine induction factors to achieve the single-turbine Betz limits.} by
\begin{equation}\label{definition of efficiency}
\eta_\ell := \mathbf{E}\left[ \frac{P_\ell}{\frac{1}{2}\rho Ax_\ell^3} \right],
\end{equation}
where $x_\ell$ is the free stream velocity entering the subarray cascaded turbines from $\ell$ to $N-1$ and $P_{\ell}$ denotes the aggregated power from the $\ell$-th subarray of wind turbines. The optimal efficiency $\eta_\ell^*$ of the $l$-th sub-array has the form
\begin{equation}\label{subarray_optimal_efficiency}
\eta_\ell^* = 4Q_\ell, \quad \forall \ell \in \{0,\ldots, N-1\},
\end{equation}
which is achieved with the optimal control sequence $u_\ell^*,\ldots,u_{N-1}^*$
, where $Q_\ell$ is calculated from \eqref{back_recursion_raw}.
\end{corollary}
\begin{proof}
The maximum power produced by the $N-\ell$ turbines is
\begin{equation}\label{optimal_subarray_power}
    P_{\ell}^* = 2\rho AQ_{\ell}x_\ell^3,
\end{equation}
under the optimal control sequence $u_\ell^*,\ldots,u_{N-1}^*$ with $Q_\ell$ computed via \eqref{back_recursion_raw}. We substitute the optimal power \eqref{optimal_subarray_power} into \eqref{definition of efficiency} and obtain \eqref{subarray_optimal_efficiency}, which concludes the proof.
\end{proof}

Next, we consider a stochastic actuator disk model with both multiplicative and additive noise, which allows a more general description of uncertainty in wind fluctuations. Interestingly, in contrast to classical linear quadratic problems, when additive noise is included the optimal policies are no longer linear in general, and so the optimal value functions are no longer cubic. This highlights a computational limitation with this more general model that makes the approach more difficult to implement in practice.

\begin{theorem}(Stochastic actuator disk model with additive noise.) Consider the stochastic ADM \eqref{stochmodel} with additive noise
\begin{equation}\label{generalized_ADM}
    x_{k+1} = a_kx_k + b_ku_k + c_k,
\end{equation}
 where $c_k \sim \mathcal{P}_c$ is a zero-mean additive random variable with second moment $\Sigma_{c,k}$ and third moment $\Gamma_{c,k}$. In the penultimate tail subproblem, the optimal policy has the nonlinear form $$\pi_{N-2}^*(x) = \delta x + \sqrt{\alpha + \beta x^2}$$ for some constants $\delta$, $\alpha$, and $\beta$; the exact expression is given in \eqref{nonlinear_control_policies}. As a result, the corresponding optimal value function at turbine location $N-2$ is non-cubic, and so the remaining optimal policies and value functions are nonlinear and non-cubic, respectively. 
\end{theorem}

\begin{proof}
Consider again the dynamic programming recursion \eqref{DP_algorithm}. Since $G_N^*(x) = 0$, the last tail subproblem is identical to that in Theorem 1, so that $G_{N-1}^*(x_{N-1}) = \frac{4}{27}x_{N-1}^3$. Consider now the penultimate tail subproblem for $k=N-2$
\begin{equation}\label{cost-to-go function_additive}
    G_{N-2}(x,u) = (x - u)^2 u + \frac{4}{27}\mathbf{E}\left[(a_k x + b_k u +c_k)^3\right].
\end{equation}
Taking the expectation of the second term by utilizing the (raw) moments of $a_k$, $b_k$ and $c_k$, and then taking the partial derivative with respect to $u$ and setting to zero yields a quadratic optimality condition in $u$. Carrying out some algebra as above, it turns out that the roots of this polynomial are no longer linear in the state, in contrast to the results in Theorem 1. The optimal control policy is thus a \emph{nonlinear} function of state of the form $\pi_{N-2}^*(x) = \delta x + \sqrt{\alpha + \beta x^2}$ for some constants $\delta$, $\alpha$, and $\beta$. The exact expression for the maximizing control input derived from the quadratic optimality condition is given in \eqref{nonlinear_control_policies}. It can also be seen that when the additive noise variance $\Sigma_{c,k}$ is zero (i.e., the additive noise is absent since it also has zero mean), then $\alpha = 0$ and we recover the linear policy of of Theorem 1 since $x\geq0$. Finally, these observations also lead to the conclusion that none of the remaining optimal policies and value functions are linear and cubic, respectively, and will in fact become increasingly complicated as the recursion proceeds backward toward the beginning of the array.
\end{proof}

\section{Numerical Experiments}
To illustrate our results, we consider a cascade with $N=10$ identical turbines to analyze the performance of the optimal gain sequence $\{\psi_0,\ldots,\psi_9\}$ for the proposed stochastic actuator disk model. As is commonly done in the literature \cite{burton2001handbook,manwell2010wind}, we refer to these gains as {\em induction factors}. The stochastic model parameters $a_k$ and $b_k$ are all independent of each other and spatially homogeneous ($\mu_{a,k} = \mu_a$, $\mu_{b,k} = \mu_b$, $\sigma_{a,k} = \sigma_{a}$, $\sigma_{b,k} = \sigma_b$, $\gamma_{a,k} = \gamma_a$ and $\gamma_{b,k} = \gamma_b$, $\forall k$)\footnote{To have clearer interpretation of our results, we discuss the results in the terms of central moments. No additive noise is considered in this section.}. 

\begin{figure}[!htbp]
\centering\label{fig:magnitude_vb}
\includegraphics[scale=0.38]{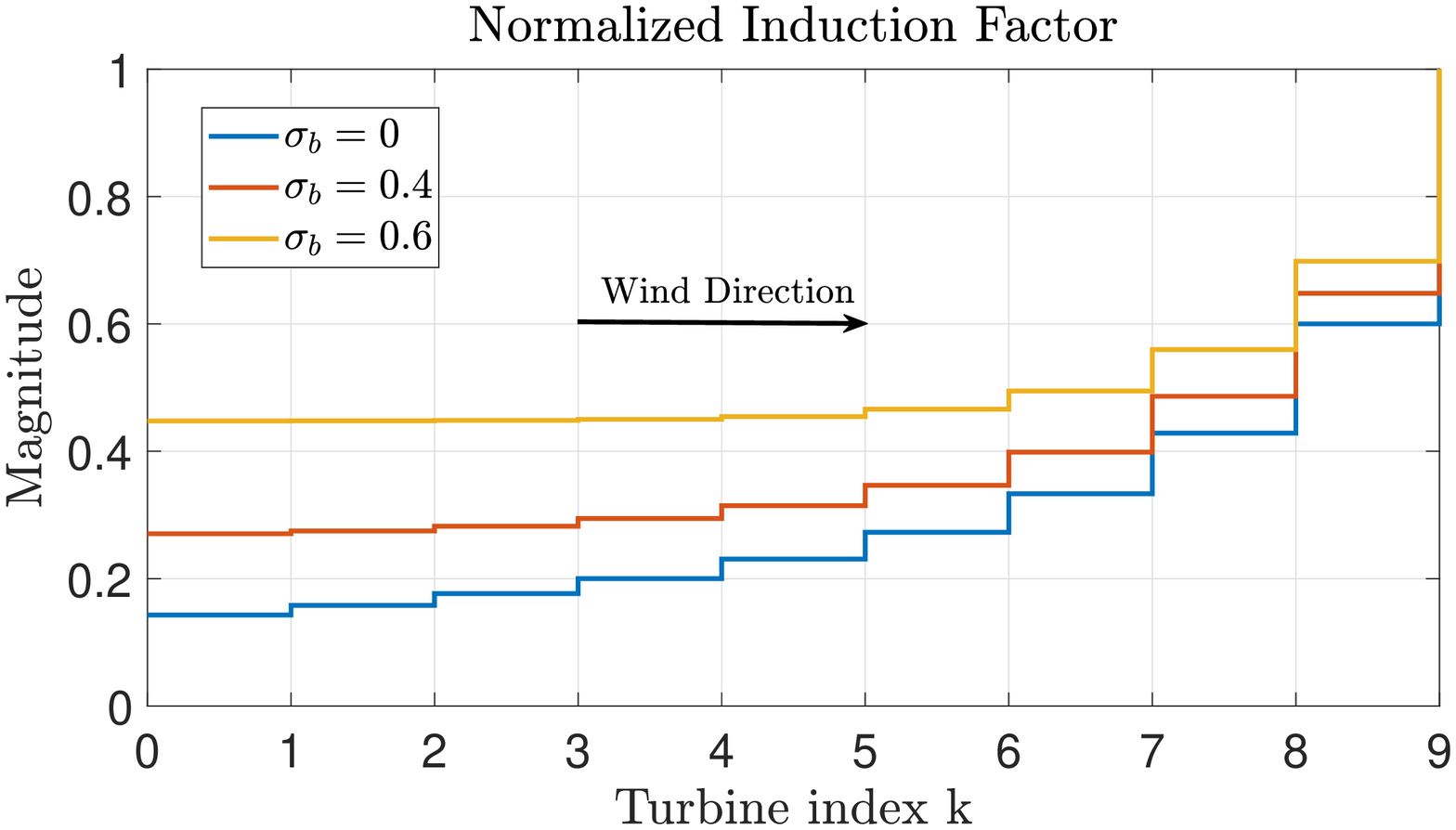}
\caption{Normalized induction factors defined as $\frac{\psi_k}{1/3}$ for deterministic model ($\mu_a = 1, \mu_b = -2$) and stochastic model with various values of {input-dependent} multiplicative noise standard deviation ($\mu_a = 1, \sigma_a = 0, \mu_b = -2, \sigma_b > 0, \gamma_a = 0$ and $\gamma_b = 0$).}
\end{figure}

\begin{figure}[!htbp]
\centering\label{fig:efficiency_vb}
\includegraphics[scale=0.37]{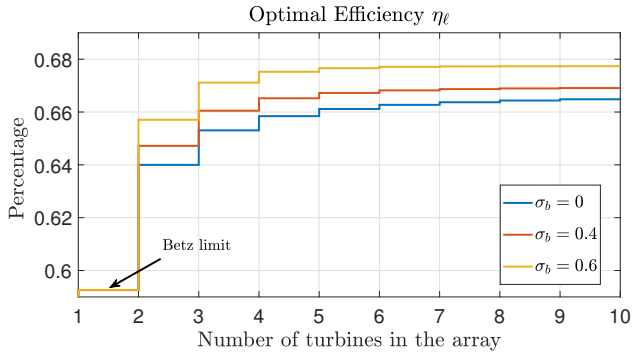}
\caption{Comparison of optimal efficiency $\eta_\ell$ for deterministic model ($\mu_a = 1, \mu_b = -2$) and stochastic model with various values of {input-dependent} multiplicative noise standard deviation ($\mu_a = 1, \sigma_a = 0, \mu_b = -2, \sigma_b > 0, \gamma_a = 0$ and $\gamma_b = 0$).}
\end{figure}

Fig. 3 illustrates the optimal induction factor sequence $\{\psi_0,\ldots,\psi_9\}$ and Fig. 4 depicts the optimal efficiency $\eta_\ell^*$ under different standard deviation values of the \emph{input-dependent} multiplicative noise $b_k$. The induction factors in Fig.~3 are normalized by $1/3$, which is the value achieving the Betz limit for a single isolated turbine \cite{burton2001handbook}. We set the mean value $\mu_a$ to 1, and the skewness to zero. Fig.~4 demonstrates that the optimal array efficiency improves with increasing variance on $b_k$. This result is intuitively reasonable, in the sense that higher variability of the velocity deficits in the far wake may lead to increased power extraction. We speculate that this multiplicative stochastic perturbation on the velocity deficit may provide a mathematically simple way of capturing physical phenomena such as mixing or entrainment, which are known to promote energy extraction \cite{verhulst2014large,santoni2015development} Note also that the case $\sigma_b=0$ reproduces the results for the deterministic ADM model in \cite{rotea2014dynamic}. It should be noted that as the standard deviation $\sigma_b$ is increased, the induction factors increase. That is, the leading upstream turbines are working more as the multiplicative noise is increased; which again is consistent with the conventional wisdom that the more turbulent the wind is the closer the turbines should be to their isolated optimum set point \cite{ciri2017large} 

%{\color{red} How does this graph supports the following: Fig. 4 also shows that the optimal deterministic control policies (derived from the model in \cite{rotea2014dynamic} with $b_k = -2$) have lower efficiency when evaluated on the stochastic actuator disk model, and it can be seen that the performance degradation becomes larger as the variance increases.}

Figures~5 and~6 provide the optimal induction factor sequence and efficiency under different standard deviation values of the \emph{state-dependent} noise on $a_k$, and without input-dependent noise (i.e., $b_k$ is fixed and constant for all $k$). Both figures demonstrate that the optimal induction factor sequence from the stochastic actuator disk model also increases the efficiency and improves performance under larger variations. 

%\color{red}Finally, the skewness of the multiplicative noise also effects the results; it can be shown that a right-skewed distribution (i.e., $\gamma > 0$) improves efficiency, whereas a left-skewed distribution (i.e., $\gamma < 0$) degrades efficiency.

To match the expected wind velocity of the conventional deterministic ADM, the mean value of the \emph{state-dependent} noise should be set to unity (i.e., $\mu_a = 1$) \cite{rotea2014dynamic}. However, having $\mu_a = 1$ together with a non-zero variance in the state-dependent multiplicative noise $a_k$ leads to null optimal induction factors for leading upstream turbines, since in this case the model essentially predicts that additional energy will be injected into the wake further downstream. This indicates that the parameters in the stochastic ADM should be carefully calibrated based on measured data in order to capture appropriate (possibly heterogeneous) spatio-temporal flow variations and obtain reasonable control policies for the array. To appropriately incorporate stochasticity of the wind flow, we set the mean value of $a_k$ to $\mu_a = 0.99$, and vary the standard deviation $\sigma_a$ to describe statistical fluctuations. 
The key observation is that regardless of the value of $\mu_a$, the proposed approach improves efficiency with increasing variance by exploiting statistical knowledge of wind field fluctuations and incorporating this information into optimal control policies for wind farm power maximization.

\color{black}

\begin{figure}[!htbp]
\centering\label{fig:magnitude_va}
\includegraphics[scale=0.38]{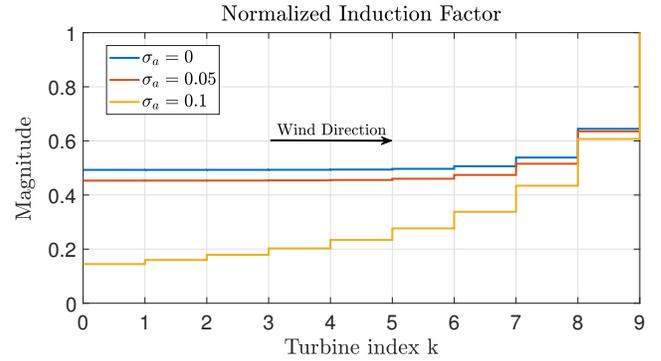}
\caption{Normalized induction factor defined as $\frac{\psi_k}{1/3}$ for deterministic model ($\mu_a = 0.99, \mu_b = -2$) and for stochastic model with various values of state-dependent multiplicative noise standard deviation ($\mu_a = 0.99, \sigma_a > 0, \mu_b = -2,  \sigma_b = 0, \gamma_a = 0$ and $\gamma_b = 0$).}
\end{figure}
\begin{figure}[!htbp]
\centering\label{fig:efficiency_va}
\includegraphics[scale=0.37]{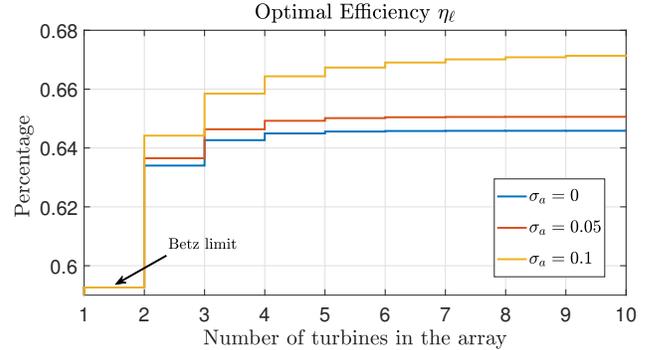}
\caption{Comparison of optimal efficiency $\eta_\ell$ for deterministic model ($\mu_a = 0.99, \mu_b = -2$) and stochastic model with various values of state-dependent multiplicative noise standard deviation ($\mu_a = 0.99, \sigma_a > 0, \mu_b = -2,  \sigma_b = 0, \gamma_a = 0$ and $\gamma_b = 0$).}
\end{figure}

The stochastic actuator disk model of a wind farm with cascaded wind turbines captures stochastic wind fluctuations. By definition, the optimal control laws derived from stochastic dynamic programming achieve superior performance to laws derived from a deterministic model of the same complexity, allowing the turbines to recognize and react to the particular wind field characteristics. Data derived directly from measurements or simulations can be incorporated directly into the control law to improve the aerodynamic efficiency a wind farm for specific wind fluctuation statistics. It is worth emphasizing that more work is necessary to incorporate tractable noise models that are consistent with the flow physics; this conference paper is a first step in this direction.

\section{Conclusions and Outlooks}
We have formulated a multi-stage stochastic optimal control problem for maximizing the power output of a one dimensional wind farm array and shown that it can be solved analytically via dynamic programming. The optimal control policies depend explicitly on the statistics of multiplicative noise, which can be related to stochastic wind fluctuations. 
%Numerical experiments demonstrate that our policies give higher efficiency than those derived from deterministic models by exploiting knowledge of the wind field fluctuations.

        Our results provide an initial step toward defining a wind farm control strategy that tractably incorporates statistical knowledge of stochastic wind fluctuations. However, there remain several lines of future work that can extend the present results in various ways to more fully understand the possibilities and limits for maximizing annual energy production. Our future work will involve
    \begin{itemize}
    \item [(a)] utilizing more realistic wake models;
    \item [(b)] estimating necessary statistics from high-fidelity numerical simulations and experimental data;
    \item [(c)] performance evaluation of the policies on high-fidelity models, which may improve the results in \cite{santoni2015development};
    \item [(d)] considering more realistic array geometries;
    \item [(e)] exploring computationally efficient approximation of nonlinear optimal control strategies; if needed. 
    \end{itemize}
\bibliographystyle{ieeetr}  
\bibliography{refs} 

\begin{thebibliography}{10}

\bibitem{pao2009tutorial}
L.~Y. Pao and K.~E. Johnson, ``A tutorial on the dynamics and control of wind
  turbines and wind farms,'' in {\em American Control Conference},
  pp.~2076--2089, June 2009.

\bibitem{xiao2018cart3}
Y.~Xiao, Y.~Li, and M.~A. Rotea, ``{CART3} field tests for wind turbine
  region-2 operation with extremum seeking controllers,'' {\em IEEE
  Transactions on Control Systems Technology}, vol.~27, no.~4, pp.~1744--1752,
  2019.

\bibitem{pao2011control}
L.~Y. Pao and K.~E. Johnson, ``Control of wind turbines,'' {\em IEEE Control
  Systems Magazine}, vol.~31, no.~2, pp.~44--62, 2011.

\bibitem{munteanu2008optimal}
I.~Munteanu, A.~I. Bratcu, N.-A. Cutululis, and E.~Ceanga, {\em Optimal Control
  of Wind Energy systems: Towards a Global Approach}.
\newblock Springer Science \& Business Media, 2008.

\bibitem{schepers2007improved}
J.~Schepers and S.~Van~der Pijl, ``Improved modelling of wake aerodynamics and
  assessment of new farm control strategies,'' in {\em Journal of Physics:
  Conference Series}, vol.~75, p.~012039, IOP Publishing, 2007.

\bibitem{johnson2009wind}
K.~E. Johnson and N.~Thomas, ``Wind farm control: Addressing the aerodynamic
  interaction among wind turbines,'' in {\em American Control Conference},
  pp.~2104--2109, June 2009.

\bibitem{knudsen2009distributed}
T.~Knudsen, T.~Bak, and M.~Soltani, ``Distributed control of large-scale
  offshore wind farms,'' in {\em European Wind Energy Conference and
  Exhibition}, Citeseer, 2009.

\bibitem{spudic2010hierarchical}
V.~Spudic, M.~Jelavic, M.~Baotic, and N.~Peric, ``Hierarchical wind farm
  control for power/load optimization,'' {\em The Science of Making Torque from
  Wind}, 2010.

\bibitem{madjidian2011distributed}
D.~Madjidian, K.~M{\aa}rtensson, and A.~Rantzer, ``A distributed power
  coordination scheme for fatigue load reduction in wind farms,'' in {\em
  American Control Conference}, pp.~5219--5224, June 2011.

\bibitem{soleimanzadeh2011controller}
M.~Soleimanzadeh and R.~Wisniewski, ``Controller design for a wind farm,
  considering both power and load aspects,'' {\em Mechatronics}, vol.~21,
  no.~4, pp.~720--727, 2011.

\bibitem{kristalny2011decentralized}
M.~Kristalny and D.~Madjidian, ``Decentralized feedforward control of wind
  farms: prospects and open problems,'' in {\em IEEE Conference on Decision and
  Control and European Control Conference}, pp.~3464--3469, Dec. 2011.

\bibitem{horvat2012quasi}
T.~Horvat, V.~Spudi{\'c}, and M.~Baoti{\'c}, ``Quasi-stationary optimal control
  for wind farm with closely spaced turbines,'' in {\em IEEE International
  Convention MIPRO}, pp.~829--834, July 2012.

\bibitem{biegel2013distributed}
B.~Biegel, D.~Madjidian, V.~Spudi{\'c}, A.~Rantzer, and J.~Stoustrup,
  ``Distributed low-complexity controller for wind power plant in derated
  operation,'' in {\em IEEE International Conference on Control Applications},
  pp.~146--151, Aug. 2013.

\bibitem{bitar2013coordinated}
E.~Bitar and P.~Seiler, ``Coordinated control of a wind turbine array for power
  maximization,'' in {\em American Control Conference}, pp.~2898--2904, June
  2013.

\bibitem{buccafusca2018maximizing}
L.~Buccafusca and C.~L. Beck, ``Maximizing power in wind turbine arrays with
  variable wind dynamics,'' in {\em IEEE Conference on Decision and Control},
  pp.~2667--2672, Dec. 2018.

\bibitem{buccafusca2017modeling}
L.~Buccafusca, C.~Beck, and G.~Dullerud, ``Modeling and maximizing power in
  wind turbine arrays,'' in {\em IEEE Conference on Control Technology and
  Applications}, pp.~773--778, Aug. 2017.

\bibitem{ciri2017model}
U.~Ciri, M.~A. Rotea, and S.~Leonardi, ``Model-free control of wind farms: A
  comparative study between individual and coordinated extremum seeking,'' {\em
  Renewable energy}, vol.~113, pp.~1033--1045, 2017.

\bibitem{santhanagopalan2018performance}
V.~Santhanagopalan, M.~Rotea, and G.~Iungo, ``Performance optimization of a
  wind turbine column for different incoming wind turbulence,'' {\em Renewable
  Energy}, vol.~116, pp.~232--243, 2018.

\bibitem{gebraad2016wind}
P.~Gebraad, F.~Teeuwisse, J.~Van~Wingerden, P.~A. Fleming, S.~Ruben, J.~Marden,
  and L.~Pao, ``Wind plant power optimization through yaw control using a
  parametric model for wake effects—a {CFD} simulation study,'' {\em Wind
  Energy}, vol.~19, no.~1, pp.~95--114, 2016.

\bibitem{goit2015optimal}
J.~P. Goit and J.~Meyers, ``Optimal control of energy extraction in wind-farm
  boundary layers,'' {\em Journal of Fluid Mechanics}, vol.~768, pp.~5--50,
  2015.

\bibitem{gebraad2015maximum}
P.~Gebraad and J.~Van~Wingerden, ``Maximum power-point tracking control for
  wind farms,'' {\em Wind Energy}, vol.~18, no.~3, pp.~429--447, 2015.

\bibitem{rotea2014dynamic}
M.~A. Rotea, ``Dynamic programming framework for wind power maximization,''
  {\em IFAC Proceedings Volumes}, vol.~47, no.~3, pp.~3639--3644, 2014.

\bibitem{johnson2012assessment}
K.~E. Johnson and G.~Fritsch, ``Assessment of extremum seeking control for wind
  farm energy production,'' {\em Wind Engineering}, vol.~36, no.~6,
  pp.~701--715, 2012.

\bibitem{yang2013maximizing}
Z.~Yang, Y.~Li, and J.~E. Seem, ``Maximizing wind farm energy capture via
  nested-loop extremum seeking control,'' in {\em ASME Dynamic Systems and
  Control Conference}, pp.~1--8, Oct. 2013.

\bibitem{marden2013model}
J.~R. Marden, S.~D. Ruben, and L.~Y. Pao, ``A model-free approach to wind farm
  control using game theoretic methods,'' {\em IEEE Transactions on Control
  Systems Technology}, vol.~21, no.~4, pp.~1207--1214, 2013.

\bibitem{park2013wind}
J.~Park, S.~Kwon, and K.~H. Law, ``Wind farm power maximization based on a
  cooperative static game approach,'' in {\em 2013 Active and Passive Smart
  Structures and Integrated Systems}, vol.~8688, p.~86880R, April 2013.

\bibitem{steinbuch1988optimal}
M.~Steinbuch, W.~De~Boer, O.~Bosgra, S.~Peeters, and J.~Ploeg, ``Optimal
  control of wind power plants,'' {\em Journal of Wind Engineering and
  Industrial Aerodynamics}, vol.~27, no.~1-3, pp.~237--246, 1988.

\bibitem{burton2001handbook}
T.~Burton, D.~Sharpe, and N.~Jenkins, {\em Handbook of Wind Energy}.
\newblock John Wiley \& Sons, 2001.

\bibitem{manwell2010wind}
J.~F. Manwell, J.~G. McGowan, and A.~L. Rogers, {\em Wind Energy Explained:
  Theory, Design and Application}.
\newblock John Wiley \& Sons, 2010.

\bibitem{bellman1954theory}
R.~Bellman, ``The {T}heory of {D}ynamic {P}rogramming,'' tech. rep., RAND Corp
  Santa Monica CA, 1954.

\bibitem{bertsekas2005dynamic}
D.~P. Bertsekas, D.~P. Bertsekas, D.~P. Bertsekas, and D.~P. Bertsekas, {\em
  Dynamic Programming and Optimal Control}, vol.~1.
\newblock Athena scientific Belmont, MA, 2005.

\bibitem{verhulst2014large}
C.~VerHulst and C.~Meneveau, ``Large eddy simulation study of the kinetic
  energy entrainment by energetic turbulent flow structures in large wind
  farms,'' {\em Physics of Fluids}, vol.~26, no.~2, p.~025113, 2014.

\bibitem{santoni2015development}
C.~Santoni, U.~Ciri, M.~Rotea, and S.~Leonardi, ``Development of a high
  fidelity {CFD} code for wind farm control,'' in {\em American Control
  Conference}, 2015.

\bibitem{ciri2017large}
U.~Ciri, M.~Rotea, C.~Santoni, and S.~Leonardi, ``Large-eddy simulations with
  extremum-seeking control for individual wind turbine power optimization,''
  {\em Wind Energy}, vol.~20, no.~9, pp.~1617--1634, 2017.

\end{thebibliography}

\end{document}